\documentclass[12pt]{article}
\usepackage{amsthm,amssymb}
\usepackage{amsmath}

\usepackage{bm}

\newcommand\az{{\mathcal Z^{\bm v}_p(s;x)}}
\newcommand\azs[1]{{\mathcal Z^{\bm v}_1(#1;x)}}
\newcommand\azbs[1]{{\mathcal Z^{\bm \beta}_1(#1;x)}}
\newcommand\azsp[1]{{\mathcal Z^{\bm v}_p(#1;x)}}
\newcommand\azx[1]{{\mathcal Z^{\bm v}_1\left(s;#1\right)}}
\newcommand\Cl{{\mbox{\,Cl}}}

{\theoremstyle{plain}%
  \newtheorem{theorem}{Theorem}
  \newtheorem{corollary}{Corollary}
  \newtheorem{proposition}{Proposition}
  \newtheorem{lemma}{Lemma}%
}
{\theoremstyle{remark}

}
{\theoremstyle{definition}
\newtheorem{definition}{Definition}

}

\begin{document}
\renewcommand{\theequation}{\arabic{equation}}
\title{Generalized Arakawa-Kaneko zeta functions}
\author{Kwang-Wu Chen\\
\small Department of Mathematics, University of Taipei,\\
\small No. $1$,  Ai-Guo West Road, Taipei $10048$, Taiwan.\\
\small E-mail: kwchen@uTaipei.edu.tw\footnote{This paper was
supported by the Ministry of Science and Technology, R.O.C.
grant under MOST 105-2115-M-845-001.}}

\maketitle

\begin{abstract}  
\setlength{\baselineskip}{14pt}
Let $p,x$ be real numbers, and $s$ be a complex number,  with 
$\Re(s)>1-r$, $p\geq 1$, and $x+1>0$.
The zeta function ${\mathcal Z}^{\bm\alpha}_p(s;x)$ is defined by
$$
{\mathcal Z}^{\bm\alpha}_p(s;x)
=\frac{1}{\Gamma(s)}\int^\infty_0
\frac{e^{-xt}}
{e^t-1}\,Li_{\bm{\alpha}}\left(\frac{1-e^{-t}}p\right)
t^{s-1}\,dt,
$$
where $\bm\alpha=(\alpha_1,\ldots,\alpha_r)$ is a $r$-tuple positive integers, and
$Li_{\bm{\alpha}}(z)$
is the one-variable multiple polylogarithms.
Since ${\cal Z}^{\bm\alpha}_1(s;0)=\xi(\bm\alpha;s)$,
we call this function as a generalized Arakawa-Kaneko zeta function.
In this paper, we investigate the properties and values of ${\cal Z}^{\bm\alpha}_p(s;x)$ 
with different values $s$, $x$, and $p$. We then give some applications on them.
\end{abstract}

\noindent{\small {\it Key Words:} 
Arakawa-Kaneko zeta functions, multiple zeta values, generalized harmonic functions,
modified Bell polynomials.

\noindent{\it Mathematics Subject Classification 2010:}
11M35, 11M41, 33B15.
}

\setlength{\baselineskip}{18pt}
\section{Introduction}
Let $p,x$ be real numbers, and $s$ be a complex number,  with 
$\Re(s)>1-r$, $p\geq 1$, and $x+1>0$.
The zeta function ${\mathcal Z}^{\bm\alpha}_p(s;x)$ is defined by
$$
{\mathcal Z}^{\bm\alpha}_p(s;x)
=\frac{1}{\Gamma(s)}\int^\infty_0
\frac{e^{-xt}}
{e^t-1}\,Li_{\bm{\alpha}}\left(\frac{1-e^{-t}}p\right)
t^{s-1}\,dt,
$$
where $\bm\alpha=(\alpha_1,\ldots,\alpha_r)$ is a $r$-tuple positive integers, and
$$
Li_{\bm{\alpha}}(z)=\sum_{1\leq n_1<n_2<\cdots<n_r}
\frac{z^{n_r}}{n_1^{\alpha_1}n_2^{\alpha_2}\cdots n_r^{\alpha_r}}
$$
is the one-variable multiple polylogarithms.

Let the generalized harmonic function $H_n^{(s)}(z)$ be defined as
$$
H_n^{(s)}(z)=\sum^n_{j=1}\frac 1{(j+z)^s},
$$
where $n\in\mathbb N$; $s\in\mathbb C$; $z\in\mathbb C\backslash\mathbb Z^-$,
$\mathbb Z^-=\{-1,-2,-3,\ldots\}$.
In particular, $H_n^{(s)}(0)=H_n^{(s)}=\sum^n_{k=1}\frac1{k^s}$ is the generalized
harmonic number, and $H_n^{(s)}(-1/2)=\sum^n_{k=1}\frac{2^s}{(2k-1)^s}=2^sO^{(s)}_n$.

The function ${\mathcal Z}^{\bm\alpha}_p(s;x)$ can be regarded as a kind of 
generalization of Arakawa-Kaneko zeta 
function. For example, ${\mathcal Z}^{\bm\alpha}_1(s;0)$ is $\xi({\bm\alpha};s)$. And if
we set $\bm{\alpha}=a$ be a positive integer, then${\mathcal Z}^a_1(s;0)=\xi_a(s)$,
where $\xi_a(s)$ is the original Arakawa-Kaneko zeta function \cite{AK}.
Moreover, 
$$
{\mathcal Z}^a_1(s;-1/2)=2^s\alpha_a(s),
\quad\mbox{and}\quad
{\mathcal Z}^a_2(s;-1/2)=2^s\beta_a(s),
$$
where $\alpha_a(s)$ and $\beta_a(s)$ are appeared
in \cite{CC1} which are defined by Coppo and Candelpergher. 
The special values $\alpha_a(s)$ and $\beta_a(s)$ can be expressed by means
of certain inverse binomial series sutdied by Kalmykov and Davydychev
in relation to the Feynman diagrams \cite{CC1, DK}.
There are a lot related works and generalizations in \cite{AK, Chen, CC, CC1, Ima, KT, Kub}.

In this paper, we evaluate ${\cal Z}_p^{\bm\alpha}(s;x)$ at $s=m+1$,
\begin{equation}\label{eq.11} 
{\mathcal Z}^{\bm\alpha}_p(m+1;x)=
\sum_{1\leq n_1<n_2<\cdots<n_r}
B\left(n_r,1+x\right)\frac{
P_m\left(H_{n_r}^{(1)}\left(x\right),
\ldots,H_{n_r}^{(m)}\left(x\right)\right)}
{p^{n_r}n_1^{\alpha_1}n_2^{\alpha_2}\cdots n_r^{\alpha_r}},
\end{equation}
where $m\in\mathbb N_0$, $\mathbb N_0=\{0,1,2,\ldots\}$,
$B(x,y)$ is the Euler beta function,
and $P_m(x_1,x_2,\ldots,x_m)$ is the modified Bell polynomial
defined by \cite{Chen, CC, CC1}
$$
\exp\left(\sum^\infty_{k=1}\frac{x_k}{k}z^k\right)=
\sum^\infty_{m=0}P_m(x_1,x_2,\ldots,x_m)z^m.
$$

For positive integers $\alpha_1,\alpha_2,\ldots,\alpha_q$, 
a multiple zeta value or $q$-fold Euler sums
of depth $q$ and weight $\varpi=\alpha_1+\alpha_2+\cdots+\alpha_q+1$
is defined as
$$
\zeta(\alpha_1,\alpha_2,\ldots,\alpha_{q-1},\alpha_q+1)
=\sum_{1\leq n_1<n_2<\cdots<n_q}
n_1^{-\alpha_1}n_2^{-\alpha_2}\cdots n_{q-1}^{-\alpha_{q-1}}n_q^{-\alpha_q-1}.
$$
For our convenience, we let $\{a\}^k$ be $k$ repetitions of 
$a$, for example $\zeta(1,\{2\}^2,4)=\zeta(1,2,2,4)$. 
Let ($a_1$, $b_1$), ($a_2$, $b_2$), $\ldots$, ($a_m$, $b_m$) 
be $m$ pairs of nonnegative integers. If we write 
$$
(\alpha_1,\alpha_2,\ldots,\alpha_q+1)
=(\{1\}^{a_1}, b_1+2, \{1\}^{a_2}, b_2+2,\ldots,\{1\}^{a_m},b_m+2)
$$
and set 
$$
(\{1\}^{b_m}, a_m+2, \{1\}^{b_{m-1}}, a_{m-1}+2,\ldots,\{1\}^{b_1},a_1+2)
=(\beta_1,\beta_2,\ldots,\beta_r+1),
$$
then the duality theorem of multiple zeta values \cite{Ohn} is stated as
$$
\zeta(\alpha_1,\alpha_2,\ldots,\alpha_q+1)
=\zeta(\beta_1,\beta_2,\ldots,\beta_r+1).
$$
Here we say that $\zeta(\beta_1,\beta_2,\ldots,\beta_r+1)$
is the dual of $\zeta(\alpha_1,\alpha_2,\ldots,\alpha_q+1)$.

This duality theorem can be used to get the special values of ${\mathcal Z}^{\bm\beta}_1(m+1;x)$.
That is to say,  if $p=1$, then the value ${\mathcal Z}^{\bm\beta}_1(m+1;x)$
have another expression:
\begin{equation}\label{eq.12} 
\azbs{m+1}=
\sum_{|\bm{d}|=m}M_q(\bm\alpha,\bm d)
\zeta(\alpha_1+d_1, \ldots, \alpha_{q-1}+d_{q-1}, 
\alpha_{q}+d_{q}+d_{q+1}+1;x),
\end{equation}
where 
$\bm d=(d_1,d_2,\ldots,d_{q+1})$ is a $(q+1)$-tuple nonegative integers,
$$
M_q(\bm\alpha,\bm d)=\prod^q_{j=1}{\alpha_j+d_j-1\choose d_j},
$$
and $\zeta(\alpha_1,\alpha_2,\ldots,\alpha_q+1;x)$ is defined by
$$
\sum_{1\leq n_1<n_2<\cdots<n_q}
(n_1+x)^{-\alpha_1}(n_2+x)^{-\alpha_2}\cdots
(n_q+x)^{-\alpha_q-1}.
$$

Combining Eq.\,(\ref{eq.11}) and Eq.\,(\ref{eq.12}) we have
\begin{eqnarray}\label{eq.13} 
\lefteqn{\sum_{1\leq n_1<n_2<\cdots<n_r}
B\left(n_r,1+x\right)\frac{
P_m\left(H_{n_r}^{(1)}\left(x\right),
\ldots,H_{n_r}^{(m)}\left(x\right)\right)}
{n_1^{\alpha_1}n_2^{\alpha_2}\cdots n_r^{\alpha_r}}}\\
&=&\nonumber
\sum_{|\bm{d}|=m}M_q(\bm\alpha,\bm d)
\zeta(\alpha_1+d_1, \ldots, \alpha_{q-1}+d_{q-1}, 
\alpha_{q}+d_{q}+d_{q+1}+1;x).
\end{eqnarray}
When $x=0$, this equation appears in \cite[Theorem A]{Chen} and
\cite[Eq.\,(21)]{CC1}.
Recently the function
\begin{eqnarray*}
\lefteqn{t(\alpha_1,\ldots,\alpha_{q-1},\alpha_q+1)}\\
&=&\sum_{1\leq n_1<n_2<\cdots<n_q}
(2n_1-1)^{-\alpha_1}(2n_2-1)^{-\alpha_2}\cdots
(2n_{q-1}-1)^{-\alpha_{q-1}}(2n_q-1)^{-\alpha_q-1}
\end{eqnarray*}
is investigated by many authors \cite{Chen1, Hof, SC, Zha1}.
It is known that 
$$
t(\alpha_1,\ldots,\alpha_{q-1},\alpha_q+1)
=2^{-|\bm\alpha|-1}\zeta(\alpha_1,\ldots,\alpha_{q-1},\alpha_q+1;-1/2).
$$
Thus we set $x=-1/2$ in Eq.\,(\ref{eq.13}), then we get a sum formula among $t$ functions
and multiple inverse binomial sums.
\begin{eqnarray*}
\lefteqn{\sum_{1\leq k_1<k_2<\cdots<k_r}
\frac{2^{2k_r}P_m(O_{k_r}^{(1)},\ldots,O_{k_r}^{(m)})}
{{2k_r\choose k_r}
k_1^{\beta_1}k_2^{\beta_2}\cdots k_r^{\beta_r+1}}} \\
&=&2^{|\bm\alpha|+1}\sum_{|\bm d|=m}
M_{q-1}(\bm\alpha,\bm d)
{\alpha_q+d_q\choose d_q}
t(\alpha_1+d_1,\ldots,\alpha_{q-1}+d_{q-1},\alpha_q+d_q+1).
\end{eqnarray*}
On the other hand, for $p\geq 2$, we have
\begin{equation}\label{eq.14} 
{\cal Z}^1_p(s;x)=\sum^\infty_{n=1}\frac{(-1)^{n+1}H_n^{(s)}(x)}{n(p-1)^n}.
\end{equation}
Note that ${\cal Z}^1_p(s;x)={\cal Z}^{\bm\alpha}_p(s;x)$, 
where $r=1$ and $\bm\alpha=\alpha=1$.
Combining Eq.\,(\ref{eq.11}) and Eq.\,(\ref{eq.14}) together, we have some interesting identities,
for example
$$
\sum^\infty_{n=1}\frac{(-1)^{n+1}O_n}{n(p-1)^n}=\arcsin(1/\sqrt{p})^2.
$$
Our paper is organized as follows. In Section 2, we present some preliminaries. In
Section 3, we prove that ${\cal Z}^{\bm\alpha}_p(s;x)$ can be analytically continued 
to an entire function, and interpolates some kind of generalized Bernoulli polynomials
at non-positive integer arguments. In Section 4, we evaluate the values 
${\cal Z}^{\bm\alpha}_p(s;x)$ at positive integers $s=m+1$, where $m\in\mathbb N_0$.
We give some applications on multiple inverse binomial sums in Section 5. 
In the final section, we use the Euler series transformation to give the evaluation
of ${\cal Z}^1_p(s;x)$ for $p\geq 2$.

\section{Preliminaries}
\begin{lemma}
Let $n\in\mathbb N$ and $x\in\mathbb R$, $z\in\mathbb C$ with $|z|<1+x$. Then
\begin{equation}\label{eq.21}   
\frac{B\left(n,1+x-z\right)}{B\left(n,1+x\right)}=
\sum^\infty_{m=0}z^m
P_m\left(H_n^{(1)}\left(x\right),
H_n^{(2)}\left(x\right),\ldots,H_n^{(m)}\left(x\right)\right),
\end{equation}
where $B(x,y)$ is the Euler beta function and 
$P_m(x_1,x_2,\ldots,x_m)$ is the modified Bell polynomial which is defined by
$$
\exp\left(\sum^\infty_{k=1}\frac{x_k}kz^k\right)
=\sum^\infty_{m=0}P_m(x_1,x_2,\ldots,x_m)z^m.
$$
\end{lemma}
\begin{proof}
Since $B(x,y)=\frac{\Gamma(x)\Gamma(y)}{\Gamma(x+y)}$,
we expand the Euler beta functions into the following form
\begin{eqnarray*}
\frac{B\left(n,1+x-z\right)}{B\left(n,1+x\right)}
&=& \prod^n_{j=1}\left(1-\frac z{j+x}\right)^{-1} \\
&=& \exp\left[-\sum^n_{j=1}\log\left(1-\frac z{j+x}\right)\right] .
\end{eqnarray*}
For $|z|<1+x$ we have
\begin{eqnarray*}
\frac{B\left(n,1+x-z\right)}{B\left(n,1+x\right)}
&=&\exp\left[\sum^n_{j=1}\sum^\infty_{k=1}
\frac 1k\left(\frac z{j+x}\right)^k\right] \\
&=& \exp\left[\sum^\infty_{k=1}\frac{z^k}k
\sum^n_{j=1}\frac 1{(j+x)^k}\right] \\
&=&\exp\left[\sum^\infty_{k=1}\frac{z^k}k
H_n^{(k)}\left(x\right)\right].
\end{eqnarray*}
From the definition of $P_m$ we know that 
$$
\frac{B\left(n,1+x-z\right)}{B\left(n,1+x\right)}
=\sum^\infty_{m=0}z^m
P_m\left(H_n^{(1)}\left(x\right),H_n^{(2)}\left(x\right),
\ldots,H_n^{(m)}\left(x\right)\right).
$$
\end{proof}

\begin{proposition}
For $n\in\mathbb N$, $m\in\mathbb N_0$, $x\in\mathbb R$
with $0<1+x$, we have
\begin{equation}\label{eq.22}   
P_m\left(H_n^{(1)}\left(x\right),
\ldots,H_n^{(m)}\left(x\right)\right)
=\frac{1}{B\left(n,1+x\right)}
\int^\infty_0 e^{-\left(1+x\right)y}(1-e^{-y})^{n-1}\frac{y^m}{m!}\,dy.
\end{equation}
\end{proposition}
\begin{proof}
For $\Re(1+x-z)>0$, we have the definition of $B(n,1+x-z)$ as
$$
B\left(n,1+x-z\right)
=\int^1_0t^{n-1}(1-t)^{x-z}\,dt.
$$
Changing the variable $t=1-e^{-y}$, we obtain
\begin{eqnarray*}
B\left(n,1+x-z\right)
&=&\int^\infty_0e^{-y\left(1+x\right)}(1-e^{-y})^{n-1}e^{yz}\,dy\\
&=&\sum^\infty_{m=0}z^m
\int^\infty_0e^{-\left(1+x\right)y}(1-e^{-y})^{n-1}
\frac{y^m}{m!}\,dy.
\end{eqnarray*}
Using the Fubini's theorem we can change the order of the integral 
and the summation in the last equation.
On the other hand, Eq.(\ref{eq.21}) gives another expression 
of $B(n,1+x-z)$. Comparing the coefficient of $z^m$, 
we conclude the result.
\end{proof}

For positive integers $\alpha_1,\alpha_2,\ldots,\alpha_q$, 
a multiple zeta value or $q$-fold Euler sums
of depth $q$ and weight $\varpi=\alpha_1+\alpha_2+\cdots+\alpha_q+1$
is defined as
$$
\zeta(\alpha_1,\alpha_2,\ldots,\alpha_{q-1},\alpha_q+1)
=\sum_{1\leq n_1<n_2<\cdots<n_q}
n_1^{-\alpha_1}n_2^{-\alpha_2}\cdots n_{q-1}^{-\alpha_{q-1}}n_q^{-\alpha_q-1}.
$$
Due to Kontsevich, we can express a multiple zeta value 
$\zeta(\alpha_1, \alpha_2, \ldots, \alpha_{q}+1)$ 
of depth $q$ and weight $\varpi=q+r$ as 
an iterated integral (or Drinfeld integrals) 
$$
\int_{E_{\varpi}} \Omega_1\Omega_2\cdots \Omega_{\varpi}
$$
with
$$
   \Omega_{j}=\left\{\begin{array}{ll} \frac{dt_j}{1-t_j}, & 
   \mbox{if}\; j=1, \alpha_1+1, \alpha_1+\alpha_2+1, \ldots,
    \alpha_1+\alpha_2+\ldots +\alpha_{q-1}+1; \\
   \frac{dt_j}{t_j}, &\mbox{otherwise}, \end{array} \right.
$$
over the simplex defined by
$$
E_{\varpi}: 0<t_1<t_2<\cdots<t_{\varpi}<1.
$$Then its dual has the iterated integral representation obtained by the change of variables:
\begin{equation}\label{eq.23} 
u_1=1-t_{\varpi}, u_2=1-t_{\varpi-1}, \ldots, u_{\varpi}=1-t_1.
\end{equation}
This leads to the identity
$$
\int_{E_{\varpi}} 
\Omega_1\Omega_2\cdots\Omega_{\varpi}=
\int_0^1 
\omega_1\omega_2\cdots\omega_{\varpi},
$$
where for $1\leq k\leq \varpi$,
$$
\omega_k=\left\{\begin{array}{ll}
\frac{du_k}{u_k},&\mbox{ if }\Omega_{\varpi+1-k}=
\frac{dt_{\varpi+1-k}}{1-t_{\varpi+1-k}},\\
\frac{du_k}{1-u_k},&\mbox{ if }\Omega_{\varpi+1-k}=
\frac{dt_{\varpi+1-k}}{t_{\varpi+1-k}}.
\end{array}\right.
$$
In particular, we have $\omega_1=du_1/(1-u_1)$ and 
$\omega_{\varpi}=du_{\varpi}/u_{\varpi}$.
If we use $\zeta(\beta_1,\beta_2,\ldots,\beta_{r}+1)$
to represent the resulting integral. Then
$$
\zeta(\alpha_1, \alpha_2, \ldots, \alpha_{q}+1)=
\zeta(\beta_1,\beta_2,\ldots,\beta_{r}+1),
$$
this is the duality theorem from another viewpoint \cite{Chen}.

Throught out this paper we always use the following statements:
Let $q$, $r$ be a pair of positive integers,
and $\zeta(\alpha_1, \alpha_2, \ldots, \alpha_{q}+1)$ 
be a multiple zeta value of depth $q$ and weight $\varpi=q+r$ 
with its dual $\zeta(\beta_1,\beta_2,\ldots,\beta_{r}+1)$
of depth $r$.

Let us consider a special multiple Hurwitz zeta function defined by
$$
\zeta(\alpha_1,\alpha_2,\ldots,\alpha_q+1;x)
=\sum_{1\leq n_1<n_2<\cdots<n_q}
(n_1+x)^{-\alpha_1}(n_2+x)^{-\alpha_2}\cdots
(n_q+x)^{-\alpha_q-1}.
$$
Its iterated integral representation is
$$
\int_{E_{\varpi}}t_1^{\,x}
\Omega_1\Omega_2\cdots\Omega_{\varpi}.
$$
Using the same change of variables as Eq.(\ref{eq.23}) 
and the evaluation of the beta function,
we can perform it as following summation.
$$
\int_{E_{\varpi}}t_1^{\,x}
\Omega_1\Omega_2\cdots\Omega_{\varpi}
=\sum_{1\leq k_1<k_2<\cdots<k_{r}}\frac{B(x+1,k_r)}
{k_1^{\beta_1}k_2^{\beta_2}\cdots k_{r}^{\beta_{r}}}
$$
under the condition $x>-1$ (ref. \cite[Proposition 1]{Chen}). 
Substituting $x$ to $x-z$ and applying Lemma 1, we have
\begin{equation}\label{eq.24}
\int_{E_{\varpi}}t_1^{\,x-z}
\Omega_1\Omega_2\cdots\Omega_{\varpi}
=\sum_{1\leq k_1<k_2<\cdots<k_{r}}B\left(1+x,k_r\right)
\frac{
{\displaystyle\sum^\infty_{m=0}
z^m}
P_m\left(H^{(1)}_{k_r}\left(x\right),\ldots,
H^{(m)}_{k_r}\left(x\right)\right)}
{k_1^{\beta_1}k_2^{\beta_2}\cdots
k_{r}^{\beta_{r}}}.
\end{equation}
That is to say we can express $\zeta(\alpha_1,\ldots,\alpha_q+1;x-z)$
as a power series of $z$ which coefficients are some multiple Euler sums.
\begin{proposition}
Let $q$, $r$ be a pair of positive integers,
and $\zeta(\alpha_1, \alpha_2, \ldots, \alpha_{q}+1)$ 
be a multiple zeta value of depth $q$ and weight $\varpi=q+r$ 
with its dual $\zeta(\beta_1,\beta_2,\ldots,\beta_{r}+1)$
of depth $r$. For $|z|<x+1$,
$$
\zeta(\alpha_1,\ldots,\alpha_q+1;x-z)=
\sum^\infty_{m=0}z^m\left(\sum_{1\leq k_1<\cdots<k_r}
\frac{B(1+x,k_r)}{k_1^{\beta_1}\cdots k_r^{\beta_r}}
P_m(H^{(1)}_{k_r}(x),\ldots,H^{(m)}_{k_r}(x))\right).
$$
\end{proposition}

\section{Zeta functions $\az$}
\begin{definition}
Let $\bm{v}=(v_1,v_2,\ldots,v_k)$ be a $k$-tuple positive integers,
$p,x$ be real numbers, and $s$ be a complex number,  with 
$\Re(s)>1-k$, $p\geq 1$, and $x+1>0$.
Then we define the zeta function $\az$ as
$$
\az
=\frac{1}{\Gamma(s)}\int^\infty_0
\frac{e^{-xt}}
{e^t-1}\,Li_{\bm{v}}\left(\frac{1-e^{-t}}p\right)
t^{s-1}\,dt,
$$
where 
$$
Li_{\bm{v}}(z)=\sum_{1\leq n_1<n_2<\cdots<n_k}
\frac{z^{n_k}}{n_1^{v_1}n_2^{v_2}\cdots n_k^{v_k}}
$$
is the one-variable multiple polylogarithms. We also define the polynomials 
$B_{p,m}^{\bm v}(x)$ as 
\begin{equation}\label{eq.31} 
\frac{e^{xt}}{e^t-1}Li_{\bm v}\left(\frac{1-e^{-t}}p\right)
=\sum^\infty_{m=0}B_{p,m}^{\bm v}(x)\frac{t^m}{m!}.
\end{equation}
\end{definition}
The function $\azx{0}$ is $\xi({\bm v};s)$ which is defined in \cite{AK}.
Let $\bm{v}=v$ be a positive integer. Then 
$$
\azx{-\frac12}=2^s\alpha_v(s),
\quad\mbox{and}\quad
{\mathcal Z}^{\bm v}_2\left(s;-\frac 12\right)=2^s\beta_v(s),
$$
where $\alpha_v(s)$ and $\beta_v(s)$ are appeared
in \cite{CC1} which are defined by Coppo and Candelpergher. 
Clearly $\azx{0}=\xi_v(s)$,
where $\xi_v(s)$ is the Arakawa-Kaneko zeta function \cite{AK}.
Moreover $B_{1,m}^{\bm v}(x)$ is the multi-poly-Bernoulli polynomial $C_m^{\bm v}(x)$
which is defined by Imatomi \cite{Ima}.

We now show that the function $\az$ can be analytically 
continued to an entire function, and interpolates $B_{p,m}^{\bm v}(x)$ 
at non-positive integer arguments. 

It is known that $Li_{\bm v}(z)$ is holomorphic for 
$z\in\mathbb C\backslash[1,\infty)$. 
Since $p\geq 1$, and $(1-e^{-t})/p\in[1,\infty)$ is equivalent to
$\Im(t)=(2j+1)\pi$ for some $j\in\mathbb Z$.
Therefore we have 
\begin{lemma}
The function $Li_{\bm v}\left(\frac{1-e^{-t}}p\right)$ is 
holomorphic for $t\in\mathbb C$ with $|\Im(t)|<\pi$.
\end{lemma}

From \cite[Lemma 2.2 (ii)]{KT}, we have 
$Li_{\bm v}(1-e^t)=O(t^k)$ as $t\rightarrow 0^+$ and
$Li_{\bm v}(1-e^t)=O(t^{v_1+\cdots+v_k})$ as $t\rightarrow\infty$.
For $t\in\mathbb R^+$, we have
\begin{eqnarray*}
\left|Li_{\bm v}\left(\frac{1-e^{-t}}p\right)\right|
&\leq& \sum_{1\leq n_1<\cdots<n_k}\left|\frac{(1-e^{-t})^{n_k}}
{n_1^{v_1}\cdots n_k^{v_k}p^{n_k}}\right| \\
&\leq& \sum_{1\leq n_1<\cdots<n_k}\left|\frac{e^{-n_kt}(e^t-1)^{n_k}}
{n_1^{v_1}\cdots n_k^{v_k}p^{n_k}}\right| 
\leq\sum_{1\leq n_1<\cdots<n_k}\left|\frac{(e^t-1)^{n_k}}
{n_1^{v_1}\cdots n_k^{v_k}}\right|.
\end{eqnarray*} 
Hence $Li_{\bm v}\left(\frac{1-e^{-t}}p\right)$ has the same 
estimates as $Li_{\bm v}(1-e^t)$. Thus we conclude that the following results.
\begin{lemma}
Let ${\bm v}=(v_1,v_2,\ldots,v_k)$ be a $k$-tuple positive integers
and $p\geq 1$.
For $t\in\mathbb R^+$, we have the estimates
$Li_{\bm v}\left(\frac{1-e^{-t}}p\right)=O(t^k)$ as $t\rightarrow 0^+$
and $Li_{\bm v}\left(\frac{1-e^{-t}}p\right)=O(t^{v_1+\cdots+v_k})$
as $t\rightarrow \infty$.
\end{lemma}
Base on a similar method in \cite[Theorem 2.3]{KT}. We give the proof of 
the function $\az$  can be analytically continued to whole $s$-plane.
\begin{proposition}
The function $\az$ can be analytically continued to whole $s$-plane
as an entire function. And the values of $\az$ at non-positive integers
are given by 
$$
\azsp{-m}=(-1)^mB_{p,m}^{\bm v}(-x).
$$
\end{proposition}
\begin{proof}
Let $\mathcal C$ be the standard contour, namely the path consisting of the positive 
real axis from the infinity to (sufficiently small) $\varepsilon$ (`top side'), 
a counter clockwise circle $\mathcal C_\varepsilon$ around the origin of radius $\varepsilon$,
and the positive real axis from $\varepsilon$ to the infinity (`bottom side').
Let 
\begin{eqnarray*}
H^{\bm v}_p(s;x) &=& \int_{\cal C}
\frac{e^{-xt}}{e^t-1}\,Li_{\bm{v}}\left(\frac{1-e^{-t}}p\right)t^{s-1}\,dt\\
&=& (e^{2\pi is}-1)\int^\infty_{\varepsilon}
\frac{e^{-xt}}{e^t-1}\,Li_{\bm{v}}\left(\frac{1-e^{-t}}p\right)t^{s-1}\,dt\\
&&\qquad\qquad+\int_{\cal C_\varepsilon}
\frac{e^{-xt}}{e^t-1}\,Li_{\bm{v}}\left(\frac{1-e^{-t}}p\right)t^{s-1}\,dt.
\end{eqnarray*}
It follows from the above lemma that $H^{\bm v}_p(s;x)$
is entire, because the integrand has no singularity on $\mathcal C$
and the contour integral is absolutely convergent for all $s\in\mathbb C$.
Suppose $\Re(s)>1-k$. The last integral tends to $0$ 
as $\varepsilon\rightarrow 0$. Hence
$$
\az=\frac1{(e^{2\pi is}-1)\Gamma(s)}H^{\bm v}_p(s;x),
$$
which can be analytically continued to $\mathbb C$, and is entire.
In fact, $\az$ is holomorphic for $\Re(s)>0$, hence no singularity
at any positive integer. Let $s=-m\leq 0$ and by Eq.\,(\ref{eq.31}),
we have
\begin{eqnarray*}
\azs{-m} &=& \frac{(-1)^mm!}{2\pi i}H^{\bm v}_p(-m;x) \\
&=& \frac{(-1)^mm!}{2\pi i}\int_{\mathcal C_{\varepsilon}}
t^{-m-1}\sum^\infty_{n=0}B_{p,n}^{\bm v}(-x)\frac{t^n}{n!}\,dt
=(-1)^mB_{p,m}^{\bm v}(-x).
\end{eqnarray*}
This completes our proof.
\end{proof}
\section{Values at positive integers $s$ in $\az$}
\begin{theorem}
Let $\bm{v}=(v_1,v_2,\ldots,v_k)$ be a $k$-tuple positive integers,
$m\in\mathbb N_0$, $p,x$ be real numbers with $p\geq 1$, $x+1>0$. Then
$$
\azsp{m+1}=
\sum_{1\leq n_1<n_2<\cdots<n_k}
B\left(n_k,1+x\right)\frac{
P_m\left(H_{n_k}^{(1)}\left(x\right),
\ldots,H_{n_k}^{(m)}\left(x\right)\right)}
{p^{n_k}n_1^{v_1}n_2^{v_2}\cdots n_k^{v_k}}.
$$
\end{theorem}
\begin{proof}
Let $s=m+1$ in the definition of $\az$,
and change the order of the integral and the summation, we have
$$
\azsp{m+1}
=\sum_{1\leq n_1<n_2<\cdots<n_k}
\frac 1{p^{n_k}n_1^{v_1}n_2^{v_2}\cdots n_k^{v_k}}
\int^\infty_0e^{-(1+x)t}(1-e^{-t})^{n_k-1}\frac{t^m}{m!}\,dt.
$$
Applying Eq.(\ref{eq.22}), we conclude the result. 
\end{proof}

The following theorem we express $\azs{m+1}$
as a linear combination of $\zeta(\alpha_1,\ldots,\alpha_q,\alpha_q+1;x)$.

\begin{theorem}
Let $q$, $r$ be a pair of positive integers,
and $\zeta(\alpha_1, \alpha_2, \ldots, \alpha_{q}+1)$ 
be a multiple zeta value of depth $q$ and weight $\varpi=q+r$ 
with its dual $\zeta(\beta_1,\beta_2,\ldots,\beta_{r}+1)$
of depth $r$.
Then for $m\in\mathbb N_0$, $x+1>0$, 
and $\bm\beta=(\beta_1,\beta_2,\ldots,\beta_r)$, we have 
$$
\azbs{m+1}=
\sum_{|\bm{d}|=m}M_q(\bm\alpha,\bm d)
\zeta(\alpha_1+d_1, \ldots, \alpha_{q-1}+d_{q-1}, 
\alpha_{q}+d_{q}+d_{q+1}+1;x),
$$
where 
$\bm d=(d_1,d_2,\ldots,d_{q+1})$ is a $(q+1)$-tuple nonegative integers.
\end{theorem}
\begin{proof}
Multiply the factor $1/m!$ and differentiate both sides with respect 
to $z$ for $m$ times to Eq.(\ref{eq.24}), we obtain the following identity.
\begin{eqnarray*}
\lefteqn{\frac{1}{m!}\int_{E_{\varpi}}t_1^{\,x-z}
\left(\log\frac 1{t_1}\right)^m\Omega_1\Omega_2\cdots \Omega_{\varpi}}\\
&=&\sum_{1\leq k_1<k_2<\cdots<k_{r}}
B\left(1+x,k_r\right)\frac{
{\displaystyle\sum^\infty_{p=0}
{p+m\choose p}z^p}
P_{p+m}\left(H^{(1)}_{k_r}\left(x\right),\ldots,
H^{(p+m)}_{k_r}\left(x\right)\right)}
{k_1^{\beta_1}k_2^{\beta_2}\cdots
k_{r}^{\beta_{r}}}.
\end{eqnarray*}

We use another integral representation 
of $\zeta(\alpha_1,\ldots,\alpha_{q-1},\alpha_q+1)$
(ref. \cite[page 120--122]{Eie}):
$$
\int_{E_{q+1}}\left\{\prod^q_{j=1}\frac 1{(\alpha_j-1)!}
\left(\log\frac{t_{j+1}}{t_j}\right)^{\alpha_j-1}\frac{dt_j}{1-t_j}\right\}\frac{dt_{q+1}}{t_{q+1}},
$$
therefore the left-hand side of the above identity becomes
$$
\frac 1{m!}\int_{E_{q+1}}t_1^{x-z}\left(\log\frac 1{t_1}\right)^m
\int_{E_{q+1}}\left\{\prod^q_{j=1}\frac 1{(\alpha_j-1)!}
\left(\log\frac{t_{j+1}}{t_j}\right)^{\alpha_j-1}\frac{dt_j}{1-t_j}\right\}\frac{dt_{q+1}}{t_{q+1}}.
$$
Replace the factor
$$
\left(\log\frac 1{t_1}\right)^{m}
=\left(\log\frac{t_2}{t_1}+\log\frac{t_3}{t_2}+\cdots +\log
\frac 1{t_{q+1}}\right)^m
$$
and then substitute by its multinomial expansion, so that the integral becomes
$$
\sum_{|\bm{d}|=m}\int_{E_{q+1}}t_1^{x-z}
\left\{\prod_{j=1}^{q}\frac{1}{(\alpha_j-1)!d_j!}
\left(\log\frac{t_{j+1}}{t_j}\right)^{\alpha_j+d_j-1}\frac{dt_j}{1-t_j}\right\}
\frac 1{d_{q+1}!}\left(\log\frac 1{t_{q+1}}\right)^{d_{q+1}}\frac{dt_{q+1}}{t_{q+1}}.
$$
In terms of the summation form, it is
$$
\sum_{|\bm{d}|=m}\prod_{j=1}^{q}{{\alpha_j+d_j-1}\choose{d_j}}
\zeta(\alpha_1+d_1, \ldots, \alpha_{q-1}+d_{q-1},
 \alpha_{q}+d_{q}+d_{q+1}+1;x-z).
$$

Now we combine these results together, we obtain
\begin{eqnarray*}
\lefteqn{
\sum_{|\bm{d}|=m}
M_q(\bm\alpha,\bm d)
\zeta(\alpha_1+d_1, \ldots, \alpha_{q-1}+d_{q-1}, 
\alpha_{q}+d_{q}+d_{q+1}+1;x-z)
}\\
&=&\sum_{1\leq k_1<k_2<\cdots<k_{r}}
B\left(1+x,k_r\right)\frac{
{\displaystyle\sum^\infty_{p=0}
{p+m\choose p}z^p}
P_{p+m}\left(H^{(1)}_{k_r}\left(x\right),\ldots,
H^{(p+m)}_{k_r}\left(x\right)\right)}
{k_1^{\beta_1}k_2^{\beta_2}\cdots
k_{r}^{\beta_{r}}}.
\end{eqnarray*}
Setting $z=0$ in the above identity, the right-hand side of identity
will become
$$
\sum_{1\leq k_1<k_2<\cdots<k_r}
B\left(k_r,1+x\right)\frac{
P_m\left(H_{k_r}^{(1)}\left(x\right),
\ldots,H_{k_r}^{(m)}\left(x\right)\right)}
{k_1^{\beta_1}k_2^{\beta_2}\cdots k_r^{\beta_r}}.
$$
This is exactly $\azbs{m+1}$, hence we get our conclusion.
\end{proof}

The formula of $\azbs{m+1}$ can
be simplified to another form.

\begin{corollary}
Let $q$, $r$ be a pair of positive integers,
and $\zeta(\alpha_1, \alpha_2, \ldots, \alpha_{q}+1)$ 
be a multiple zeta value of depth $q$ and weight $\varpi=q+r$ 
with its dual $\zeta(\beta_1,\beta_2,\ldots,\beta_{r}+1)$
of depth $r$.
Then for $m\in\mathbb N_0$, $x+1>0$, 
and $\bm\beta=(\beta_1,\beta_2,\ldots,\beta_r)$, we have 
$$
\azbs{m+1}=
\sum_{|\bm{d}|=m}M_{q-1}(\bm\alpha,\bm d){\alpha_q+d_q\choose d_q}
\zeta(\alpha_1+d_1, \ldots, \alpha_{q-1}+d_{q-1}, 
\alpha_{q}+d_{q}+1;x),
$$
where 
$\bm d=(d_1,d_2,\ldots,d_q)$ is a $q$-tuple nonegative integers.
\end{corollary}
\begin{proof}
Changing the variables in the summation $(d_1,d_2,\ldots,d_q,d_{q+1})$
to $(a_1,a_2,\ldots,a_q,a_{q+1})$. with $a_i=d_i$ for $i=1,2,\ldots,q-1$,
$a_q=d_q+d_{q+1}$, $a_{q+1}=d_{q}$, and then using the identity
$$
\sum^{a_q}_{a_{q+1}=0}{\alpha_q+a_{q+1}-1\choose a_{q+1}}
={\alpha_q+a_q\choose a_q},
$$
the final formula will be gotten.
\end{proof}
Let $p=1$ in Theorem 1. Then we combine the result in Corollary 1, we have the following
identity.
\begin{theorem}
Let $q$, $r$ be a pair of positive integers,
and $\zeta(\alpha_1, \alpha_2, \ldots, \alpha_{q}+1)$ 
be a multiple zeta value of depth $q$ and weight $\varpi=q+r$ 
with its dual $\zeta(\beta_1,\beta_2,\ldots,\beta_{r}+1)$
of depth $r$.
Then for $m\in\mathbb N_0$, $x+1>0$, 
and $\bm\beta=(\beta_1,\beta_2,\ldots,\beta_r)$, we have 
\begin{eqnarray*}
\lefteqn{\sum_{1\leq n_1<n_2<\cdots<n_r}
B\left(n_r,1+x\right)\frac{
P_m\left(H_{n_r}^{(1)}\left(x\right),
\ldots,H_{n_r}^{(m)}\left(x\right)\right)}
{n_1^{\beta_1}n_2^{\beta_2}\cdots n_r^{\beta_r}}}\\
&=&
\sum_{|\bm{d}|=m}M_{q-1}(\bm\alpha,\bm d){\alpha_q+d_q\choose d_q}
\zeta(\alpha_1+d_1, \ldots, \alpha_{q-1}+d_{q-1}, 
\alpha_{q}+d_{q}+1;x),
\end{eqnarray*}
where 
$\bm d=(d_1,d_2,\ldots,d_q)$ is a $q$-tuple nonegative integers.
\end{theorem}

\section{Multiple inverse binomial sums}
Note that we always use the following statements
throught out this paper :
Let $q$, $r$ be a pair of positive integers,
and $\zeta(\alpha_1, \alpha_2, \ldots, \alpha_{q}+1)$ 
be a multiple zeta value of depth $q$ and weight $\varpi=q+r$ 
with its dual $\zeta(\beta_1,\beta_2,\ldots,\beta_{r}+1)$
of depth $r$.

First we let $x=0$ in Theorem 3, we get a formula concerning Euler sums.
\begin{eqnarray*}
\lefteqn{\sum_{1\leq k_1<k_2<\cdots<k_r}
\frac{P_m(H_{k_r}^{(1)},\ldots,H_{k_r}^{(m)})}
{k_1^{\beta_1}k_2^{\beta_2}\cdots k_{r-1}^{\beta_{r-1}}k_r^{\beta_r+1}}} \\
&=&\sum_{|\bm d|=m}
M_{q-1}(\bm\alpha,\bm d)
{\alpha_q+d_q\choose d_q}
\zeta(\alpha_1+d_1,\ldots,\alpha_{q-1}+d_{q-1},\alpha_q+d_q+1).
\end{eqnarray*}
This identity first appears in \cite[Theorem A]{Chen}. 
Let $r=1$ in the above identity, we get an identity with the Arakawa-Kaneko zeta values
\cite[Eq.\,(4)]{CC}: for $m\geq 0$,
$$
\xi_q(m+1)=\sum^\infty_{n=1}\frac{P_m(H_n^{(1)},\ldots,H_n^{(m)})}{n^{q+1}}.
$$

We let $x=-1/2$ and $s=m+1$ with $m\in\mathbb N_0$. 
We could see that
$$
H_n^{(s)}\left(-\frac 12\right)=\sum^n_{k=1}\frac 1{(k-\frac 12)^s}
=\sum^n_{k=1}\frac{2^s}{(2k-1)^s}
=2^s\,O^{(s)}_n.
$$
And for any number $a$, 
$$
P_m(ax_1,a^2x_2,\ldots,a^mx_m)=
a^mP_m(x_1,x_2,\ldots,x_m).
$$
Therefore if we apply these values in Theorem 1, we have 
$$
{\mathcal Z}^{\bm v}_p(m+1;-1/2)
=\sum_{1\leq n_1<n_2<\cdots<n_k}
\frac{2^{m+2n_k}P_m(O_{n_k}^{(1)},\ldots,O_{n_k}^{(m)})}
{{2n_k\choose n_k}p^{n_k}
n_1^{v_1}n_2^{v_2}\cdots n_k^{v_k+1}},
$$
where ${\bm v}=(v_1,v_2,\ldots,v_k)\in\mathbb N^k$ and 
${\bm\beta}=(\beta_1,\beta_2,\ldots,\beta_r)\in\mathbb N^r$.
Since ${\mathcal Z}^v_1(m+1;-1/2)=2^{m+1}\alpha_v(m+1)$ and
${\mathcal Z}^v_2(m+1;-1/2)=2^{m+1}\beta_v(m+1)$, 
we can get Eq.(6) and Eq.(7) in \cite{CC1}:
\begin{eqnarray*}
2^{m+1}\alpha_v(m+1) &=& {\mathcal Z}^v_1(m+1;-1/2)
= \sum^\infty_{n=1}\frac{2^{m+2n}}{{2n\choose n}n^{v+1}}
P_m(O^{(1)}_n,\ldots,O^{(m)}_n),\\
2^{m+1}\beta_v(m+1) &=& {\mathcal Z}^v_2(m+1;-1/2)
= \sum^\infty_{n=1}\frac{2^{m+n}}{{2n\choose n}n^{v+1}}
P_m(O^{(1)}_n,\ldots,O^{(m)}_n).
\end{eqnarray*}

Applying Theorem 3 with $x=-1/2$,
then we will get a general formula involving $O_n^{(s)}$.
\begin{eqnarray}
\nonumber
\lefteqn{\sum_{1\leq k_1<k_2<\cdots<k_r}
\frac{2^{m+2k_r}P_m(O_{k_r}^{(1)},\ldots,O_{k_r}^{(m)})}
{{2k_r\choose k_r}
k_1^{\beta_1}k_2^{\beta_2}\cdots k_r^{\beta_r+1}}} \\
&=&\sum_{|\bm d|=m} \label{eq.51}
M_{q-1}(\bm\alpha,\bm d)
{\alpha_q+d_q\choose d_q}
\zeta(\alpha_1+d_1,\ldots,\alpha_{q-1}+d_{q-1},\alpha_q+d_q+1;-1/2).
\end{eqnarray}
In view of 
$$
t(\alpha_1,\ldots,\alpha_{q-1},\alpha_q+1)
=2^{-|\bm\alpha|-1}\zeta(\alpha_1,\ldots,\alpha_{q-1},\alpha_q+1;-1/2),
$$
we rewrite Eq.\,(\ref{eq.51}) as 
\begin{eqnarray}\nonumber
\lefteqn{\sum_{1\leq k_1<k_2<\cdots<k_r}
\frac{2^{2k_r}P_m(O_{k_r}^{(1)},\ldots,O_{k_r}^{(m)})}
{{2k_r\choose k_r}
k_1^{\beta_1}k_2^{\beta_2}\cdots k_r^{\beta_r+1}}} \\
&=&2^{|\bm\alpha|+1}\sum_{|\bm d|=m}\label{eq.53} 
M_{q-1}(\bm\alpha,\bm d)
{\alpha_q+d_q\choose d_q}
t(\alpha_1+d_1,\ldots,\alpha_{q-1}+d_{q-1},\alpha_q+d_q+1).
\end{eqnarray}

We only write two special cases for this identity as its applications. 
Firstly, we consider $q=1$.
Thus $\alpha=r$ and $\beta_1=\beta_2=\cdots=\beta_r=1$. 
The right-hand side of Eq.\,(\ref{eq.51}) becomes
\begin{eqnarray*}
{r+m\choose m}\zeta(r+m+1;-1/2)
&=& {r+m\choose m}\sum^\infty_{n=1}\frac 1{(n-\frac 12)^{r+m+1}} \\
&=& {r+m\choose m}(2^{r+m+1}-1)\zeta(r+m+1).
\end{eqnarray*}
Therefore we have the following identity.
\begin{corollary}
For $m\in\mathbb N_0$, $r\in\mathbb N$, we have
\begin{equation}\label{eq.52} 
\zeta(r+m+1)=\frac 1{{r+m\choose m}(2^{r+m+1}-1)}
\sum_{1\leq k_1<k_2<\cdots<k_r}
\frac{2^{m+2k_r}P_m(O_{k_r}^{(1)},\ldots,O_{k_r}^{(m)})}
{{2k_r\choose k_r}k_1\cdots k_{r-1}k_r^2}.
\end{equation}
\end{corollary}
Let $r=1$ in this identity, it will give \cite[Eq.\,(8)]{CC1}.
That is to say, let $r=m=1$ in the above identity, we will regain
the formula for Ap\'{e}ry's constant (ref \cite[Eq.(9)]{CC1}):
$$
\sum^\infty_{n=1}\frac{2^{2n}}{{2n\choose n}}
\frac{O_n}{n^2}=7\zeta(3).
$$

Moreover, if we set $r=2$ in Eq.\,(\ref{eq.52}), then equation becomes
\begin{eqnarray*}
\zeta(m+3) &=& 
\frac 1{{m+2\choose m}(2^{m+3}-1)}
\sum_{1\leq k_1<k_2}
\frac{2^{m+2k_2}P_m(O_{k_2}^{(1)},\ldots,O_{k_2}^{(m)})}
{{2k_2\choose k_2}k_1 k_2^2} \\
&=& \frac{2^{m+1}}{(m+1)(m+2)(2^{m+3}-1)}
\sum_{k_2=2}^\infty\sum^{k_2-1}_{k_1=1}
\frac{2^{2k_2}P_m(O_{k_2}^{(1)},\ldots,O_{k_2}^{(m)})}
{{2k_2\choose k_2}k_1 k_2^2} \\
&=& \frac{2^{m+1}}{(m+1)(m+2)(2^{m+3}-1)}
\sum_{n=1}^\infty
\frac{2^{2n}H_{n-1}P_m(O_{n}^{(1)},\ldots,O_{n}^{(m)})}
{n^2{2n\choose n}}.
\end{eqnarray*}
We conclude as the following corollary.
\begin{corollary}
For $m\in\mathbb N_0$, we have
$$
\zeta(m+3) =
\frac{2^{m+1}}{(m+1)(m+2)(2^{m+3}-1)}
\sum_{n=1}^\infty
\frac{2^{2n}H_{n-1}P_m(O_{n}^{(1)},\ldots,O_{n}^{(m)})}
{n^2{2n\choose n}}.
$$
\end{corollary}
We list some values of this formula:
\begin{enumerate}
\item Let $m=0$, then 
$$
7\zeta(3)=\sum^\infty_{n=1}\frac{2^{2n}}{{2n\choose n}}
\frac{H_{n-1}}{n^2}.
$$
This formula is different from the formula for 
Ap\'{e}ry's constant. But it can be found 
in \cite[Eq.(2.36)]{DK} with $u=4$ and $\theta=\pi$. 
\item Let $m=1$, then
$$
45\,\zeta(4)=\sum^\infty_{n=1}
\frac{2^{2n+1}}{{2n\choose n}}
\frac{H_{n-1}O_n}{n^2}.
$$
Since $O_n=H_{2n-1}-H_{n-1}/2$,
we can solve this formula from \cite[Eq.(2.38),Eq.(2.39)]{DK}
with $u=4$ and $\theta=\pi$.
\item Let $m=2$, then we have
$$
93\,\zeta(5)=\sum^\infty_{n=1}\frac{2^{2n}}{{2n\choose n}}
\frac{H_{n-1}(O_n^2+O_n^{(2)})}{n^2}.
$$
\end{enumerate}

Secondly, we set $r=1$. Thus $\beta=q$ and $\alpha_1=\alpha_2=\cdots=\alpha_q=1$.
Therefore Eq.\,(\ref{eq.53}) becomes the following style.
\begin{corollary}
For $m\in\mathbb N_0$ and $q\in\mathbb N$, we have
$$
\sum^\infty_{n=1}\frac{2^{2n}P_m(O^{(1)}_n,\ldots,O^{(m)}_n)}
{{2n\choose n}(2n)^{q+1}}
=\sum_{|\bm d|=m}(d_q+1)\,t(d_1+1,\ldots,d_{q-1}+1,d_q+2).
$$
\end{corollary}
We list some values of this formula:
\begin{enumerate}
\item Let $m=0$, then
$$
t(\{1\}^{q-1},2)=\sum^\infty_{n=1}\frac{2^{2n}}{{2n\choose n}(2n)^{q+1}}.
$$
\item Let $m=1$, then we have
$$2\,t(\{1\}^{q-1},3)+\sum_{a+b=q-2}t(\{1\}^a,2,\{1\}^b,2)
=\sum^\infty_{n=1}\frac{2^{2n}O_n}{{2n\choose n}(2n)^{q+1}}.
$$
\end{enumerate}

\section{Integral Transform}
The Laplace transformation ${\cal L}(f)$ of $f$ is defined by
$$
{\cal L}(f)(u)=\int^\infty_0e^{-ut}f(t)\,dt,
$$
if the integral is convergent. Define the operators $D$ and $S$
as the following integral transforms:
\begin{eqnarray*}
D(f)(u) &=& \int^\infty_0\frac{(1-e^{-t})^u}{e^t-1}f(t)\,dt,\\
S(f)(u) &=& \int^\infty_0\frac{1-e^{-ut}}{e^t-1}f(t)\,dt.
\end{eqnarray*}
These two operators $D$ and $S$ were introduced in \cite{CaC, CC}.
We list a basic property of them which is usually called ``the Euler series transformation''.
\begin{lemma}\cite[Equation 4]{CC}
For any complex numbers $z$ such that $|z|<1/2$, we have
$$
\sum^\infty_{k=1}D(f)(k)\frac{z^k}k
= -\sum^\infty_{k=1}\frac1kS(f)(k)\left(\frac{z}{z-1}\right)^k.
$$
\end{lemma}
It is interesting that these operators are connected with the evaluation
of ${\cal Z}^{\bm\alpha}_p(s;x)$.
\begin{proposition}
Let
$$
\lambda_{s,t}(x)=\frac{e^{-xt}t^{s-1}}{\Gamma(s)}.
$$
Then
\begin{equation}\label{eq.61}
{\cal Z}_p^{\bm\alpha}(s;x)
=\sum_{1\leq n_1<\cdots<n_r}\frac{D(\lambda_{s,x})(n_r)}
{n_1^{\alpha_1}\cdots n_r^{\alpha_r}p^{n_r}}.
\end{equation}
\end{proposition}
\begin{proof}
\begin{eqnarray*}{\cal Z}_p^{\bm\alpha}(s;x) &=&\frac1{\Gamma(s)}\int^\infty_0
\frac{e^{-xt}}{e^t-1}Li_{\bm\alpha}\left(\frac{1-e^{-t}}p\right)t^{s-1}\,dt \\
&=& \frac{1}{\Gamma(s)}\int^\infty_0\frac{e^{-xt}}{e^t-1}
\sum_{1\leq n_1<\cdots<n_r}\frac{(1-e^{-t})^{n_r}}{n_1^{\alpha_1}\cdots n_r^{\alpha_r}p^{n_r}}
t^{s-1}\,dt\\
&=&\sum_{1\leq n_1<\cdots<n_r}\frac1{n_1^{\alpha_1}\cdots n_r^{\alpha_r}p^{n_r}}
\int^\infty_0\frac{(1-e^{-t})^{n_r}}{e^t-1}\lambda_{s,x}(t)\,dt\\
&=& \sum_{1\leq n_1<\cdots<n_r}\frac{D(\lambda_{s,x})(n_r)}{n_1^{\alpha_1}
\cdots n_r^{\alpha_r}p^{n_r}}.
\end{eqnarray*}
\end{proof}
Therefore we are interested in evaluating the values
of $S(\lambda_{s,x})(n)$ and $D(\lambda_{s,x})(n)$ 
for a positive integer $n$.

The Laplace transformation of $\lambda_{st}(x)$ is
$$
{\cal L}(\lambda_{s,t})(y)=\frac1{\Gamma(s)}
\int^\infty_0e^{-(x+y)t}t^{s-1}\,dt
=\frac1{(x+y)^s}.
$$
Applying the operator $S$ and $D$ to $\lambda_{s,t}$ we have
\begin{eqnarray*}
S(\lambda_{s,t})(n) &=& \frac1{\Gamma(s)}\int^\infty_0
\frac{1-e^{-nt}}{1-e^{-t}}e^{-(1+x)t}t^{s-1}\,dt \\
&=& \sum^{n-1}_{k=0}\frac1{\Gamma(s)}\int^\infty_0
e^{-(k+1+x)t}t^{s-1}\,dt\\
&=&\sum^{n-1}_{k=0}\frac1{(k+1+x)^s}=H^{(s)}_n(x).
\end{eqnarray*}
And
\begin{eqnarray*}
D(\lambda_{s,x})(n) &=& \frac1{\Gamma(s)}
\int^\infty_0(1-e^{-t})^{n-1}e^{-(1+x)t}t^{s-1}\,dt \\
&=& \sum^{n-1}_{k=0}(-1)^k{n-1\choose k}\frac1{\Gamma(s)}
\int^\infty_0e^{-(k+1+x)t}t^{s-1}\,dt\\
&=& \sum^{n-1}_{k=0}(-1)^k{n-1\choose k}(k+1+x)^{-s}.
\end{eqnarray*}
We conclude the above results as follows.
\begin{proposition}
\begin{eqnarray*}
{\cal L}(\lambda_{s,x})(y) &=& (x+y)^{-s},\\
S(\lambda_{s,x})(n) &=& H_n^{(s)}(x),\\
D(\lambda_{s,x})(n) &=& \sum^{n-1}_{k=0}(-1)^k{n-1\choose k}(x+k+1)^{-s}.
\end{eqnarray*}
\end{proposition}
In particular, if we set $s=m+1$, for $m$ is a nonnegative integer, we have
\begin{eqnarray*}
D(\lambda_{m+1,x})(n) &=& \int^\infty_0\frac{(1-e^{-t})^n}{e^t-1}\frac{e^{-xt}t^m}{m!}\,dt\\
&=& \int^\infty_0e^{-(x+1)t}(1-e^{-t})^{n-1}\frac{t^m}{m!}\,dt
\end{eqnarray*}
In view of Eq.\,(\ref{eq.22}), we found 
\begin{proposition}
$D(\lambda_{m+1,x})(n)=B(n,1+x)P_m(H_n^{(1)}(x),\ldots, H_n^{(m)}(x)$.
\end{proposition}
Applying this result in Eq.\,(\ref{eq.61}) we get the equation in Theorem 1 again.
In the following we give some applications for the special values
$r=1$ and ${\bm\alpha}=\alpha=1$ to ${\cal Z}^{\bm\alpha}_p(s;x)$. 
That is to say, we consider the different representations of 
${\cal Z}^1_p(s;x)$.
\begin{proposition}
Let $p\geq 2$, we have
\begin{equation}\label{eq.64} 
\sum^\infty_{n=1}\frac{D(\lambda_{s,x})(n)}{np^n}
={\cal Z}^1_p(s;x)=\sum^\infty_{n=1}\frac{(-1)^{n+1}H_n^{(s)}(x)}{n(p-1)^n}.
\end{equation}
In particular, if $s=m+1$, then
\begin{equation}\label{eq.62} 
\sum^\infty_{n=1}\frac{B(n,1+x)}{np^n}P_m(H_n^{(1)}(x),\ldots,H_n^{(m)}(x))
=\sum^\infty_{n=1}\frac{(-1)^{n+1}H_n^{(m+1)}(x)}{n(p-1)^n}.
\end{equation}
\end{proposition}
\begin{proof}
Using Lemma 4 with $z=1/p$ we get Eq.\,(\ref{eq.64}) for $p>2$. 
Since 
$$
\sum^\infty_{n=1}\frac{D(\lambda_{s,x})(n)}{n2^n}
={\cal Z}^1_2(s;x)
$$
is convergent, then by the classical Abel lemma,
$$
\sum^\infty_{n=1}\frac{D(\lambda_{s,x})(n)}{n2^n}
=\sum^\infty_{n=1}\frac{(-1)^{n+1}H_n^{(s)}(x)}n.
$$
Therefore Eq.\,(\ref{eq.64}) is true for $p=2$.
\end{proof}
Let $x=0$ in Eq.\,(\ref{eq.62}), we have the following identity
$$
\sum^\infty_{n=1}\frac{P_m(H_n^{(1)},\ldots,H_n^{(m)})}
{n^2p^n}=\sum^\infty_{n=1}\frac{(-1)^{n+1}H_n^{(m+1)}}{n(p-1)^n}.
$$
\cite[Eq.\,(16)]{CC1} is the special case with $p=2$ in the above 
identity. In the following we list some identities with $x=-1/2$.
\begin{equation}\label{eq.63}
\sum^\infty_{n=1}\frac{2^{2n-1}}{n^2p^n{2n\choose n}}
P_m(O_n^{(1)},\ldots,O_n^{(m)})
=\sum^\infty_{n=1}\frac{(-1)^{n+1}O_n^{(m+1)}}{n(p-1)^n}.
\end{equation}
Let $\theta=2\arcsin(1/\sqrt{p})$. Then set $m=0$ in the above identitiy, 
we have
$$
2\sum^\infty_{n=1}\frac{(-1)^{n+1}O_n}{n(p-1)^n}
= \sum^\infty_{n=1}\frac{4^n}{n^2p^n{2n\choose n}}
= \frac{\theta^2}2.
$$
The right-most equation is derived by \cite[Eq.\,(C.2)]{DK} with
$u=4/p$. In the following we list some identities with special values of $p$:
$$\begin{array}{rclcrcl}
p&=&4&\Rightarrow&\displaystyle\sum^\infty_{n=1}\frac{(-1)^{n+1}O_n}{n3^n} 
&=& \displaystyle\frac{\pi^2}{36},\\
p&=&2&\Rightarrow&\displaystyle\sum^\infty_{n=1}\frac{(-1)^{n+1}O_n}{n} 
&=& \displaystyle\frac{\pi^2}{16},\\
p&=&8+4\sqrt{3}&\Rightarrow&\displaystyle\sum^\infty_{n=1}\frac{(-1)^{n+1}O_n}{n(7+4\sqrt{3})^n} 
&=& \displaystyle\frac{\pi^2}{144},\\
p&=&6+2\sqrt{5}&\Rightarrow&\displaystyle\sum^\infty_{n=1}\frac{(-1)^{n+1}O_n}{n(5+2\sqrt{5})^n} 
&=& \displaystyle\frac{\pi^2}{100},\\
p&=&\displaystyle\frac{2(5+\sqrt{5})}{5}&\Rightarrow
&\displaystyle\sum^\infty_{n=1}\frac{(-1)^{n+1}5^nO_n}{n(5+2\sqrt{5})^n} 
&=& \displaystyle\frac{\pi^2}{25}.
\end{array}$$
Let $m=1$ in Eq.\,(\ref{eq.63}), we have
\begin{eqnarray*}
2\sum^\infty_{n=1}\frac{(-1)^{n+1}O_n^{(2)}}{n(p-1)^n}
&=&\sum^\infty_{n=1}\frac{4^nO_n}{n^2p^n{2n\choose n}} \\
&=& -2\Cl_3(\theta)+2\Cl_3(\pi-\theta)-\theta\Cl_2(\pi-\theta)
-\theta\Cl_2(\theta)+\frac72\zeta(3),
\end{eqnarray*}
where $\Cl_n(\theta)$ is the Clausen's function and the last 
equation is derived by \cite[Eq.\,(2.36), (2.37)]{DK}
with $u=4/p$.
%
%
%

\end{document}